\documentclass[10pt]{amsart}
\usepackage{amsthm,amssymb,latexsym,xypic,amsmath}
   \usepackage[latin1]{inputenc}
\def \C {\mathbb{C}}

\def \F {{\mathcal F}}

\newtheorem{proposition}{Proposition}[section]
\newtheorem{definition}[proposition]{Definition}
\newtheorem{cor}[proposition]{Corollary}
\newtheorem{lema}[proposition]{Lemma}
\newtheorem{teo}[proposition]{Theorem}
\newtheorem{example}[proposition]{Example}
\newtheorem{re}[proposition]{Remark}
 
\begin{document}

\title[ Classification  of  holomorphic Pfaff systems  on Hopf manifolds]
{Classification  of  holomorphic Pfaff systems  on Hopf manifolds}

\author{Maur\'icio Corr\^ea*}
\address{Maur\'icio Corr\^ea \\ ICEx - UFMG \\
Departamento de Matem\'atica \\
Av. Ant\^onio Carlos 6627 \\
30123-970 Belo Horizonte MG, Brazil } \email{mauriciojr@ufmg.br}
\thanks{*Corresponding author.}

\author{Antonio M. Ferreira}
\address{ Antonio M. Ferreira \\ DEX - UFLA , Av P.H. Rolfs, s/n, Campus Universit\'ario, Lavras MG, Brazil, CEP 37200-000}
\email{antoniosilva@dex.ufla.br}

\author{Misha Verbitsky}
\address{ Misha Verbitsky \\ IMPA-Instituto Nacional de Matem\'atica Pura e Aplicada, Estrada Dona Castorina, 110,
Jardim Bot\^anico, CEP 22460-320, Rio de Janeiro, RJ -Brazil\\
 and \\
 Laboratory of Algebraic Geometry, \\
  National Research University Higher School of Economics,\\
 6 Usacheva Str., Moscow, Russia.\\ Misha Verbitsky is supported by
 Russian Academic Excellence Project '5-100'',
 CNPq - Process 313608/2017-2, and FAPERJ E-26/202.912/2018.
}
\email{verbit@impa.br}


\begin{abstract}
We classify  holomorphic Pfaff systems (possibly non locally decomposable)   on
certain Hopf manifolds. As  consequence,  we  prove  some   integrability results. 
We also prove that any holomorphic distribution on 
a general (non-resonance) Hopf manifold  is integrable. 
\end{abstract}
\maketitle
 
\section{Introduction}
Let $W=\mathbb{C}^n-\{0\}$, $ n\geq 2$ and $f(z_1, z_2,\dots ,
z_n)=(\mu_1 z_1,\mu_2 z_2,\dots ,\mu_n z_n)$ be a diagonal contraction 
in $\mathbb{C}^n$, where $0<|\mu_i|<1$ for
all $1\leq i\leq n$. The  quotient space
$X=W/\langle f\rangle  $ is a compact complex manifold of dimension $n$ called a diagonal 
Hopf manifold. When $\mu_1=\dots=\mu_{n}$ we say that $X$ is a classical Hopf manifold.
Classical Hopf manifolds were first studied by Heinz Hopf \cite{hopf} in 1948. Hopf showed that $X$ is diffeomorphic to the product of odd spheres $S^1\times S^{2n-1}$ and has
a complex structure which is not-K\"ahler. K. Kodaira \cite{_Kodaira:surf_II_} classified  Hopf surfaces. The   geometry  of  Hopf 
manifolds have been studied by several authors, see for instance, Dabrowski \cite{Da}, Haefliger \cite{Hae}, Kato \cite{_Kato:subvarieties_}, Ise \cite{Ise}, Wehler\cite{weh} etc.

A  holomorphic singular  Pfaff system  $\F$, of  codimension $k$, on $X$ is a non-trivial section $ \omega \in \mathrm H^0(X,\Omega_X^k\otimes  \mathcal{L})$,  where $ \mathcal{L}$ is a  holomorphic line bundle on $X$. 
In this work we will study the geometry and classification of    holomorphic Pfaff systems on Hopf manifolds.

 We will consider  the following types of Hopf manifolds: \footnote{ In the work \cite{aca}  the authors have used other terminology. }  

\begin{definition}  
We say that
\begin{enumerate}
\item $X$ is \textbf{ classical } if $\mu=\mu_1=\cdots=\mu_n$.
\item $X$ is \textbf{ no-resonance} if  $0<|\mu_1|\leq|\mu_2|\leq\ldots\leq|\mu_n|<1$
 and there is no  non-trivial relation among the $\mu_i$'s in this way
$$\prod_{i\in A}\mu^{r_{i}}_{i}=\prod_{j\in B}\mu^{r_{j}}_{j},\quad r_i,r_j
\in\mathbb{N},\quad A\cap B=\emptyset,\quad A\cup B=\{1,2,\ldots,n\}.$$

\item $X$ is \textbf{ weak no-resonance } if $\mu_1=\mu_2=\cdots=\mu_r$, where 
$2\leq r\leq n-1$ and there is no  non-trivial relation among the $\mu_i$'s in this way
$$\prod_{i\in A}\mu^{r_{i}}_{i}=\prod_{j\in B}\mu^{r_{j}}_{j},
\quad r_i,r_j\in\mathbb{N},\quad A\cap B=\emptyset,\quad A\cup B=\{1,r+1,\ldots,n\}.$$
\end{enumerate}
\end{definition}

\begin{re} Note that a general Hopf manifold is always no-resonance.
Indeed, each of the algebraic conditions $\prod_{i\in A}\mu^{r_{i}}_{i}=\prod_{j\in B}\mu^{r_{j}}_{j}$
gives a closed algebraic subvariety of positive codimension 
in the space of all operators $A$. Therefore, the set of no-resonance
Hopf manifold is obtained as a complement to a countable union
of such subvarieties.
\end{re}

D. Mall in \cite{Ma1} used the Kodaira's classification of Hopf surfaces 
to obtain a complete description of  regular  holomorphic foliations on 
Hopf surfaces. E. Ghys in \cite{ghys},  studied 
regular holomorphic foliations on homogeneous spaces, and as a consequence obtained the classification of codimension one foliation on classical Hopf manifolds.

In \cite{aca} M.  Corr\^ea, A. Fern\'andez-P\'erez,  and A. M. Ferreira   classified regular holomorphic  foliations of dimension and codimension one on certain Hopf manifolds.  
As a consequence of the classification   the authors  proved that any  regular 
codimension one distribution on a weak no-resonance or no-resonance Hopf manifold is
integrable and admits a holomorphic first integral.

 M. Ise proved in  \cite{Ise} that if $X$ is a  classical Hopf manifold, then a  line bundle $L$ on $X$ is the quotient of  $W\times \mathbb{C}$  by the
operation of a representation of the fundamental group
of  $X$ 
$
  \varrho_{L}:   \pi_1(X)\simeq \mathbb{Z}  \longrightarrow   GL(1,\mathbb{C})  =  \mathbb{C}^*
$ in the following way
 \begin{equation}\label{_character_Equation_}
\begin{array}{ccc}
  W\times \mathbb{C}& \longrightarrow &    W\times \mathbb{C}  \\
 (z,v)& \longmapsto  &     (f(z),  \varrho_{L}(1)v)\
\end{array}
\end{equation}
In \cite{Ma} M. Mall  has generalized this fact for arbitrary Hopf manifolds. 
We will write $L=L_b$ for the line bundle induced by the representation $
\varrho_{L}(\gamma)$ with  $b=\varrho_{L}(1)$. 

\hfill

In this paper we will show   the following:
\begin{teo}\label{teo1}
Let $X$ be a Hopf manifold  and let
$\mathcal{F}$ be a singular  holomorphic Pfaff system   
on $X$, of codimension $k $,    given
by a nonzero twisted
differential $k$-form $\omega\in H^0(X,\Omega^{k}_X\otimes
L_b)$ with coefficients in the line bundle
$L_b$.   Consider the natural projection $\pi: W \to X$. 
Then the following holds:
\begin{itemize}
\item[(i)] If $X$ is classical, then  $b =\mu^{m}$ with  $m 
\in \mathbb{N}$ and $m\geq
k$. Furthermore $\pi^*\omega$ is   a   homogeneous polynomial $k$-form
of  degree  $m-k$. 
\\
\item[(ii)]
 If $X$ is  no-resonance, then  $b =\mu_{1}^{m_1}\dots\mu_{n}^{m_n}$ such that $m_{j}\in\{0,1\}$ for all $j=1,\dots, n$.
Furthermore,  $\pi^*\omega$  is a monomial $k$-form of the type
$$ \sum\limits _{i_1<\dots<i_k} g_{i_1\dots i_k} dz_{i_1}\wedge \dots \wedge dz_{i_k},$$
with
$$
g_{i_1\dots  i_k}(z_1,\dots  ,z_n) =  c_{ m} ^{{i_1\dots  i_k}}  z_{j_1} ^{m_{j_1}}z_{j_2} ^{m_{j_2}}
\dots z_{j_{n-k}}^{m_{j_{n-k}}}  
$$
where $\{j_1,j_2,\dots, j_{n-k}\}  =\{1,\dots, n\} \setminus \{i_1,\dots, i_k\}$. 
\\
\item[(iii)] If $X$ is  weak no-resonance, then $b=\mu^{m}  \mu_{r+1}^{m_{r+1}}\dots \mu_n^{m_n}$ with
$m \in \mathbb{N}$,  and    $m_{j}\in\{0,1\}$ for all $j=r+1,\dots, n$.
Furthermore, $\pi^*\omega$ is a   polinomial  k-form of the type
$$\pi^*\omega=\sum\limits _{i_1<\dots<i_k} g_{i_1\dots i_k} dz_{i_1}\wedge \dots \wedge dz_{i_k},$$ where 
$$
g_{i_1\dots i_k}=z_{j_1} ^{m_{j_1}}\dots z_{j_{n-r-k+s}} ^{m_{j_{n-r-k+s}}} \cdot  
\sum\limits _{ \alpha_1+\dots +\alpha_r=m-s} c_{\alpha}^{i_1\dots  i_k} z_1 ^{\alpha_1}\dots z_r ^{\alpha_r}
$$
such that  $\{j_1,\dots,j_{n-r-k+s}\}=\{r+1, \dots,n\}\setminus \{i_{s+1}, \dots,i_k\}$ and $s\in \{0,\dots,k\}$ is defined   in such a way that  $i_1,\dots, i_s\leq r$ and   $r+1\leq i_{s+1},\dots, i_k .$
\end{itemize}
\end{teo}

\hfill

The part $i)$ of the Theorem \ref{teo1} generalizes the  result in \cite{ghys}  in the  codimension one case and \cite{aca} in the one-dimensional case.

We obtain the following consequence for regular Pfaff system on no-resonance Hopf manifolds. 

\begin{cor}  \label{cor-gen}
Any regular holomorphic Pfaff system of codimension $k$   on a no-resonance Hopf manifold is integrable and   has a  compact leaf.
\end{cor}

For singular holomorphic distributions we show the following integrability result.

\begin{teo}\label{integrab}
Let $\mathcal{F}$ be 
a singular holomorphic distribution on a   no-resonance   Hopf
manifold  $X$. Then $\mathcal{F}$ is  integrable. 
\end{teo}

\subsection*{Acknowledgments}
Maur\'icio Corr\^ea  is supported by the CNPq grant numbers 202374/2018-1, 302075/2015-1, 400821/2016-8. 
Misha Verbitsky  is supported by Russian Academic Excellence Project '5-100'', CNPq - Process 313608/2017-2, and FAPERJ E-26/202.912/2018.
 We would like to thank the
referee for precious comments which improved the presentation of the paper greatly.

\section{Pfaff systems, Distributions and Holomorphic foliations}\label{Pfaff-sec}
Let $X$ be an $n$-dimensional complex manifold.

\begin{definition} 
A  holomorphic singular \textbf{Pfaff system} $\F$, of  codimension $k$, on $X$ is a non-trivial section $ \omega \in \mathrm H^0(X,\Omega_X^k\otimes  \mathcal{L})$,  where $ \mathcal{L}$ is a  holomorphic line bundle on $X$.    The singular set of $\omega$ is defined by  
 $$\mathrm{Sing}(\F):=\mathrm{Sing}(\omega)=\{z\in X; \ \omega(z)=0\}.$$
We suppose that $  \mathrm{cod}( \mathrm{Sing}(\omega))\geq 2$.  We say that  $\omega$ is \textbf{regular} (or non-singular ) if $\mathrm{Sing}(\omega)=\emptyset$.
\end{definition}

\begin{re}
 Sometimes a Pfaff system is defined as a saturated
 rank 1 subsheaf in $\Omega_X^k$. This definition 
 is not equivalent to the one given above because
 of ambiguity the choice of the line bundle $\mathcal{L}$.
 However, for most practical purposes there notions are
 the same.
 \end{re}

Given a Pfaff system $\F$       on $X$ induced by  $ \omega \in \mathrm H^0(X,\Omega_X^k\otimes  \mathcal{L})$, then $\omega$ is determined by the following: 
\begin{itemize}
\item [(i)] a open covering $\{U_{\alpha}\}_{\alpha\in \Lambda}$  of $X$;
\item [(ii)] holomorphic $k$-forms $\omega_{\alpha} \in \Omega^k_{U_{\alpha}}$ satisfying  $
\omega_{\alpha } = h_{\alpha \beta} \omega_{\beta} $ on 
 $U_{\alpha}\cap U_{\beta}\neq \emptyset$,\end{itemize}
\noindent where $h_{\alpha \beta}\in \mathcal{O}(U_{\alpha}\cap U_{\beta})^{\ast}$ determines the cocycle representing $ \mathcal{L}$.  Therefore, the singular set of $\F$ is   $ \cup_{\alpha \in \Lambda }  \{\omega_{\alpha} =0\} $.  For more details about Pfaff systems see  \cite{EK,CJV,CMS}. 

A   \textbf{ distribution} on $X$ of codimension $k$  is a nonzero coherent subsheaf  $\mathcal{G}\subset \Omega_X^1$ of rank $k$ 
 such that $\mathcal{G}$ is saturated in $ \Omega_X^1$ (i.e., $ \Omega_X^1/\mathcal{G}$ is torsion free).
 
 The $k$-th wedge product of the inclusion   $\mathcal{G}\subset \Omega_X^1$ gives rise to a   twisted holomorphic  differential $k$-form 
with coefficients in the line bundle $ \det(\mathcal{G})$. That is, a distribution of codimension $k$ induces a Pfaff system  
$$ \omega \in \mathrm H^0(X,\Omega_X^k\otimes   \det(\mathcal{G}))$$
  of codimension $k$ which is  {\it locally decomposable outside its    singular set}. 
That is,  for each point $p\in X\setminus  \mathrm{Sing}(\omega)$ there exists a neighbourhood  $U$ and
holomorphics $1$-forms $\omega_1, \dots,\omega_k \in H^0(U, \Omega_U^1)$ such that
$$
\omega|_{U}=\omega_1 \wedge \cdots   \wedge \omega_k.
$$
A singular holomorphic distribution $\mathcal{F}$ is a    \textbf{ foliation} if it is   \textbf{ integrable}, i.e,  if
$$
d \omega_i \wedge \omega_1 \wedge \cdots   \wedge \omega_k=0
$$
for all $i=1,\dots,k$.
The tangent sheaf of a distribution $\F$ induced by a twisted form $\omega$ is 
$$
T\F=\{v\in TX;\ i_v\omega=0\}.
$$
The rank of $T\F$ is  $(n-k)$ and by integrability condition we have that    $[T\F,T\F]\subset T\F$.

There is the  following exact
sequence of sheaves 
\[
0 \to T\mathcal F \to TX \to N_{\mathcal F} \to 0 \,
\]
where $ N_{\mathcal F}$ is a torsion free sheaf of rank $k$. Conversely, given a nonzero coherent subsheaf $T\mathcal F  \subset TX$ of rank $k$
 such that $TX/T\F=N_{\F}$ is torsion free defines a distribution $(N_{\F})^*\subset  \Omega_X^1.$

We remark that by construction the tangent bundle of a Hopf manifold $X$ is given
by
$$
TX=\bigoplus_{i=1}^n L_{\mu_i},
$$
where $L_{\mu_i}$ is the tangent bundle of the foliation induced
by the canonical vector field $\frac{\partial}{ \partial z_i}$.

\section{Cohomology of line bundles on Hopf manifolds}
Let $W=\mathbb{C}^n-\{0\}$, $ n\geq 2$ and $f(z_1, z_2,...,
z_n)=(\mu_1 z_1,\mu_2 z_2,...,\mu_n z_n)$ be a diagonal contraction 
in $\mathbb{C}^n$, where $0<|\mu_i|<1$ for
all $1\leq i\leq n$. Consider the associated  Hopf manifold 
$X=W/\langle f \rangle$.

Let $\Omega^{p}_{X}$ be the sheaf of germs of holomorphic $p$-forms
on a Hopf manifold $X$. Denote by $\Omega^{p}_{X}(L_b):=\Omega^{p}_{X}\otimes L_b$,
where $L_b$ is a line bundle determined by a character $b$ as in \eqref{_character_Equation_},
and by $\pi:W\to X$ the natural projection on $X$. 

We will adopt the notation in \cite{Ma}. Consider an open covering $A=\{U_{i}\}$
 of $X$ such that    $U_{i}$  are open and contractible
Stein subsets of $X$ and  $\tilde{U}_{i}:=\pi^{-1}(U_{i})$ is a disjoint union of Stein open sets $\{U_{ij}'\}$ of $W$, each of them isomorphic to $U_i$. 
We have that
$$\displaystyle\tilde{U}_{i}=\bigcup_{r\in\mathbb{Z}}f^{r}(U_{i0}); \ \ \tilde{A}=\{\tilde{U}_{i}\}.$$

\par  Consider  $\varphi\in\Gamma(U_i,\Omega^{p}_{X}(L_b))$.  Then 
$$\tilde{\varphi}=
\pi^{*}(\varphi) \in \Gamma(\tilde{U}_i,\pi^{*}(\Omega^{p}_{X}(L_b)))
\cong\Gamma(\tilde{U}_i,\Omega^{p}_{W}).$$
Denote the map  $p_0=b\cdot id-f^*: H^0 (W, \Omega_W ^p)\rightarrow H^0 (W, \Omega_W ^p)$. We have an exact sequence of C\v{e}ch complexes
\begin{equation}\label{sequencia de Cech}
 \xymatrix{ 0 \ar[r] & \mathcal{C}^\textbf{.}(A, \Omega_{X} ^p(L_b))  \ar[r]^{\pi^*}  & \mathcal{C}^\textbf{.}(\tilde{A}, \Omega^p _W)    \ar[r]^{p_0}  & \mathcal{C}^\textbf{.}(\tilde{A}, \Omega_W ^p)
 \ar[r]& 0  }
\end{equation}

From (\ref{sequencia de Cech}) we derive the long exact sequence of cohomology
\begin{eqnarray*}
 \xymatrix{ 0 \ar[r] &  H^0(X, \Omega_{X} ^p (L_b))  \ar[r]   &  H^0(W, \Omega_W ^p)   \ar[r]^{p_0}   &  H^0(W, \Omega_W ^p) \ar[r]  & H^1(X, \Omega_{X} ^p (L_b))
 \ar[r]&  \cdots}
 \end{eqnarray*}

\hfill

D. Mall proved in  \cite{Ma} the following result:

\begin{teo}[Mall \cite{Ma}] \label{teocoh}
If $X$ is a Hopf manifold, of dimension $n\geq 3$,  and $L_b$ is a line bundle on $X$, then
\begin{eqnarray*}\label{equacao99}
&&\dim H^0(X, \Omega_{X} ^p (L_b))=\dim\,Ker(p_0).
\end{eqnarray*}
\end{teo}

In order to prove Theorem \ref{teo1} we first  prove the following lemma.

\hfill

\begin{lema}\label{le1}
Let $X$ be a classical, no-resonance or   weak no-resonance Hopf manifold of dimension 
$n\geq3$, and let $L_b$ be a line bundle on $X$, with $b\in \mathbb{C}^*$. The following holds:
 \begin{enumerate}
\item[(i)] If $X$ is classical,  then $\dim\,H^0(X,\Omega_X ^k \otimes L_b)>0$ 
if and only if $b=\mu^m$, where $m\in \mathbb{N}, \, m\geq k$.

\item[(ii)] If $X$ is no-resonance, then $\dim\, H^0(X,\Omega_X ^k \otimes L_b)>0$ 
if and only if  $$b=\mu_1 ^{m_1}\mu_2 ^{m_2}\dots \mu_n ^{m_n}$$ where $m_i\in 
\mathbb{N}$ and there exist  $j_1,\dots,j_k\in \{1,\dots  ,n\}$, such that \mbox{$m_{j_1},\dots,m_{j_k}\geq 1$.}

\item[(iii)] If $X$ is weak no-resonance, then  $\dim\,H^0(X,\Omega_X ^k \otimes L_b)>0$ 
if and only if $$b=\mu_1^{m_1} \mu_2^{m_2}\dots \mu_n^{m_n}$$ with
$m_j \in \mathbb{N}$ for all $j=1,\dots ,n$,
$m_1+m_2+\dots +m_r= t$, and there exist $i_1,\dots ,i_{k-t}\geq 
r+1$ such that  $m_{i_1}\geq 1 ,\dots,  m_{i_{k-t}}\geq 1$.
 \\
\end{enumerate}
\end{lema}

\begin{proof}
From Theorem \ref{teocoh} we have that  $\dim H^0(X, \Omega_X ^{k}\otimes L_b)=\dim(ker\, p_0)$,  where the map
$$p_0: H^0 (W, \Omega_ W ^{k})\longrightarrow H^0 (W, \Omega_ W ^{k}),$$
is given by   $ p_0=b\cdot 
 id-f^*. $
Let $\omega \in H^0(W,\Omega_W ^{k})$ be.  We  write 
$$\omega= \sum\limits _{i_1<\dots<i_k}  g_{i_1\dots i_k} dz_{i_1}\wedge \cdots \wedge dz_{i_k}.$$ 
It  follows from Hartog's extension theorem that each $g_{i_1\dots i_k}$ 
can be represented by its Taylor series
\begin{center}
$g_{i_1\dots i_k}(z_1,z_2,\dots ,z_n)=\sum\limits _{\alpha \in \mathbb{N}^n} 
c_\alpha ^{i_1\dots i_k} z_1^{\alpha_1}z_2^{\alpha_2}\dots z_n ^{\alpha_n}$, for all $i=1,\dots ,n.$
\end{center}
Hence
\begin{equation}\label{equacao123}
p_0(\omega)=\sum\limits _{i_1<\dots <i_k} \sum\limits _{\alpha \in \mathbb{N}^n}
c_\alpha ^{i_1\dots i_k} (b-\mu_1 ^{\alpha_1}\dots \mu_n ^{\alpha_n}\mu_{i_1}\dots 
\mu_{i_k}) z_1^{\alpha_1} \dots z_{n} ^{\alpha_n} dz_{i_1}\wedge \cdots \wedge dz_{i_k}.
\end{equation}

In the classical case,  $\mu_1=\dots =\mu_n=\mu$ and
thus $\dim(\ker(p_0)) >0$ if, and only if,
$b=\mu^{m}$, for some $m \in \mathbb{N}$, $m\geq k$.
\\

In the no-resonance case, since $\mu_i's$ have no relations, it  follows from (\ref{equacao123}) 
that \\ \mbox{$ \dim(\ker(p_0)) >0$} if, and only if,  $b=\mu_1 ^{m_1} \mu_2^{m_2}\dots \mu_n ^{m_n}$ 
where $m_j \in \mathbb{N}$, and there exist
 $j_1,\dots,j_k\in \{1,\dots  ,n\}$, 
such that \mbox{$m_{j_1},\dots,m_{j_k}\geq 1$.}
\\

Finally, for the weak no-resonance case, we have $\mu_1=\dots =\mu_r=\mu$. Since 
 $\mu, \mu_{r+1},\dots ,\mu_n$ have no relations, we have
$\dim(\ker(p_0))>0$ if, and only if, $b=\mu^m \mu_{r+1}^{m_{r+1}}\dots \mu_n^{m_n}$ such that
$m_j \in \mathbb{N}$ for all $j=1,\dots ,n$, $m_1+m_2+\dots +m_r= t$, and there exist $i_1,\dots ,i_{k-t}\geq 
r+1$ such that $m_{i_1}\geq 1 ,\dots,  m_{i_{k-t}}\geq 1$.

\end{proof}

\subsection{Proof  of Theorem \ref{teo1}}
\begin{proof}
First of all, we may  assume that $k\leq n-2$, since the case when $k= n-1$ was done in \cite{aca}. 
Consider the natural projection $\pi: W \to X$. By construction,  we have that
a holomorphic section $\omega \in H^0(X, \Omega_X ^k \otimes L_{b})$ corresponds to a non-trivial  section
 $$\pi^* \omega=(g_{i_1\dots i_k})_{i_1<\dots < i_k} \in H^0\left(W, \mathcal{O}_W ^{n \choose k}\right),$$
 such that
$g_{i_1\dots i_k} \in \mathcal{O}_W$ satisfies
$$g_{i_1\dots i_k}(\mu_1 z_1,\dots ,\mu_n z_n)= \mu_{i_1}^{-1}\dots \mu_{i_k} ^{-1} 
b  g_{i_1\dots i_k}(z_1,\dots ,z_n),$$ for all ${i_1<\dots < i_k}.$
By Hartog's extension theorem, $\pi^* \omega$ can be represented on $ \mathbb{C}^n$ by its Taylor series
\begin{center}
$g_{i_1\dots  i_k}(z_1,\dots ,z_n)=\sum\limits _{\alpha \in \mathbb{N}^n} c_\alpha 
^{i_1\dots  i_k} z_1^{\alpha_1}\dots  z_n ^{\alpha_n}$, where $\alpha=(\alpha_1,
\alpha_2,\dots ,\alpha_n) \in \mathbb{N}^n$ .
\end{center}
Then
$$ 
c_{\alpha}^{i_1\dots  i_k} \mu_1 ^{\alpha_1} \mu_2 ^{\alpha_2}\dots  \mu_ n ^{\alpha_n}=
c_{\alpha}^{i_1\dots  i_k} \mu_{i_1}^{-1}\dots \mu_{i_k} ^{-1} b,
$$
where  $\alpha =(\alpha_1, \alpha_2, \dots ,\alpha_n)\in \mathbb{N}^n.$ Therefore,   for all  $ c_{\alpha}^{i_1\dots i_k}\neq 0$ we get that  
\begin{equation}\label{equa6}
  \mu_1 ^{\alpha_1} \mu_2 ^{\alpha_2}\dots  \mu_ n ^{\alpha_n}=
  \mu_{i_1}^{-1}\dots \mu_{i_k} ^{-1} b,
 \end{equation}
\\
\textbf{Classical case.} In this case we have that  $\mu_1=\dots = \mu_n=\mu$. Lemma \ref{le1} item $(i)$ implies that
 $b=\mu^{m}$ for some
$m\geq k$.
Therefore, from (\ref{equa6}) we have 
$$
 \mu^{|\alpha|}=  \mu^{-k}\mu^{m}, \  \textrm{ where } 
|\alpha|=\alpha_1+\dots +\alpha_n.
$$
This  implies that   $|\alpha|=m-k$  
 for all  $\alpha =(\alpha_1, \alpha_2, \dots ,\alpha_n)\in \mathbb{N}^n$ and  $i_1<\dots< i_k$ such that   $ c_{\alpha}^{i_1\dots i_k}\neq 0$. It follows that  each $g_{i_1\dots i_k}$ 
is a homogeneous polynomial of degree $m-k$.
\\ \\
\textbf{No-resonance case.} If $X$ is no-resonance, then by Lemma \ref{le1} item $(ii)$ we have
$$b=\mu_1 ^{m_1} \mu_2^{m_2}\dots \mu_n ^{m_n}$$  where $m_i\in \mathbb{N}$ 
and there exist $l_1,\dots,l_k\in \{1,\dots  ,n\}$ 
such that \mbox{$m_{l_1},\dots,m_{l_k}\geq 1$.} So from (\ref{equa6}) we get
\begin{center}
$ 
  \mu_1 ^{\alpha_1} \mu_2 ^{\alpha_2}\dots  \mu_ n ^{\alpha_n}
=  
\mu_{i_1}^{-1}\dots \mu_{i_k} ^{-1}\mu_1 ^{m_1} \mu_2 ^{m_2}\dots \mu_n ^{m_n},$ 
\end{center}
for all  $\alpha =(\alpha_1, \alpha_2, \dots ,\alpha_n)\in \mathbb{N}^n$ and  $i_1<\dots< i_k$ such that  $ c_{\alpha}^{i_1\dots i_k}\neq 0$. 
Since there is not non-trivial relation between  the $\mu_i$'s,  we must have that   $\alpha=\tilde{m}=:(m_1, \dots, m_{i_1}-1,\dots,m_{i_k}-1\dots, m_n)$. Thus 
\begin{eqnarray} \label{generica} \nonumber \\ 
g_{i_1\dots  i_k}(z_1,\dots  ,z_n) &=& c_{\tilde{m}} ^{{i_1\dots  i_k}} z_{1} ^{m_{1}}\dots z_{i_{1}} ^{m_{i_{1}}-1}\dots z_{i_k} ^{m_{i_k}-1} \dots z_n ^{m_n},
\\ \nonumber
\end{eqnarray}
  \noindent 
Therefore the $k$-form $\pi^*\omega$ is  monomial. 

We must  have  that $m_i\in\{0,1\}$, for all $i\in\{1,\cdots, n\}$,
otherwise, we   would have that  $\{z_i=0\}\subset \mathrm{Sing}(\pi^*\omega)$ for some $i$. A contradiction, since the Pfaff system has singular set with codimension $\geq 2$. Then, if the coefficient $g_{i_1\dots  i_k} \not\equiv 0$,  we have  from   (\ref{generica})   that 
\begin{eqnarray}\label{generica2} 
g_{i_1\dots  i_k}(z_1,\dots  ,z_n) &=&   c_{\tilde{m}} ^{{i_1\dots  i_k}}  z_{j_1} ^{m_{j_1}}z_{j_2} ^{m_{j_2}}
\dots z_{j_{n-k}}^{m_{j_{n-k}}}  
\end{eqnarray}
where $\{j_1,j_2,\dots, j_{n-k}\}  =\{1,\dots, n\} \setminus \{i_1,\dots, i_k\}$, and $m_{i_1}=\cdots=m_{i_k}=1$.
\\ 
Now we  will  determine  the regular  $k$-forms. 
\subsubsection{Proof  of Corollary  \ref{cor-gen}}  If $\pi^*\omega=(g_{i_1\dots i_k})_{i_1<\dots < i_k} $ is regular, then  there  exist a  coefficient $g_{h_1\cdots  h_k} $, for some $(h_1,\dots,  h_k)$,   such that  it   is not zero  along  $V_1=\cap_{i=2}^{n} \{z_i=0\}$.  We will consider the following   subcases:
\begin{enumerate}
\item suppose that  $h_1=1$:  
from equation (\ref{generica2}) we have 
\begin{eqnarray*} 
g_{h_1\dots  h_k}(z_1,\dots  ,z_n) &=&   c_{\tilde{m}} ^{{h_1\dots  h_k}}  z_{j_1} ^{m_{j_1}}z_{j_2} ^{m_{j_2}}
\dots z_{j_{n-k}}^{m_{j_{n-k}}}  
\end{eqnarray*}
with $2 \leq j_1\leq \dots \leq j_{n-k}$ and $m_{h_1}=\cdots=m_{h_k}=1$. Since $g_{h_1\cdots  h_k} $  is not zero along  $V_1$, 
we must have that $m_{j_1}=m_{j_2}=\cdots m_{j_{n-k}}=0$. Now, as we know all $m_i's$, from equation \ref{generica2}, we have that the only coefficient  of the $k$-form $\pi^*\omega$ which is not identically zero is
$g_{h_1\dots  h_k}=c_{\tilde{m}} ^{{h_1\dots  h_k}}=constant$. Therefore 
$$\pi^*\omega=c_{\tilde{m}} ^{{h_1\dots  h_k}} dz_{h_1}\wedge dz_{h_2}\wedge\cdots\wedge dz_{h_k}.$$

\item suppose that $h_1 \neq 1$:  
 from equation (\ref{generica2}) we have
\begin{eqnarray*} 
g_{h_1\dots  h_k}(z_1,\dots  ,z_n) &=&   c_{\tilde{m}} ^{{h_1\dots  h_k}}  z_{1} ^{m_1}z_{j_2} ^{m_{j_2}}
\dots z_{j_{n-k}}^{m_{j_{n-k}}}  
\end{eqnarray*} where $\{j_1=1,j_2,\dots, j_{n-k}\}  =\{1,\dots, n\} \setminus \{h_1,\dots, h_k\}$, and $m_{h_1}=\cdots=m_{h_k}=1$.

Since $g_{h_1\cdots  h_k} $  is not zero along  $V_1$, 
we must have $m_{j_2}=\cdots m_{j_{n-k}}=0$.
Now we have two subcases: \\
1a) If $m_1=0$, we are in the previous case $m_{h_1}=\cdots=m_{h_k}=1$ and $m_j=0$ in the other cases, and therefore 
$$\pi^*\omega=Cdz_{h_1}\wedge dz_{h_2}\wedge\cdots\wedge dz_{h_k}$$ 
1b) If $m_1=1$, we have  $m_{h_1}=\cdots=m_{h_k}=1$  and $m_{j_2}=\cdots m_{j_{n-k}}=0$ in the other cases.
Again, as now we know all $m_i's$, from equation \ref{generica2}, we have that the only coefficients  of the $k$-form $\pi^*\omega$ which is not a null monomial is
\begin{equation*} g_{h_1\dots h_k}=C z_1,\,\, g_{1h_2\dots  h_k}=C_1z_{h_1},
g_{1h_1h_3\cdots  h_k}=C_2z_{h_2},\cdots, g_{1h_1h_2\cdots    h_{k-1}}=C_{k}z_{h_k}\end{equation*}
Since $k+1\leq n-1$, the   $k$-form $\pi^*\omega$ which  induces the  Pfaff system    is not regular. A contradiction. 

\end{enumerate}

Therefore any   regular holomorphic Pfaff system on a no-resonance  Hopf manifold  $X=W/\langle f\rangle  $   is induced by a constant  $k$-form of the type
$$\pi^*\omega=Cdz_{i_1}\wedge\dots\wedge dz_{{i_k}},\ C\in \mathbb{C}^*.$$ 
That is, it is   always integrable and has a  compact leaf $W_k/ \langle f\rangle $, where $$W_k=\{z_{i_1}=\dots =z_{{i_k}}=0\}-\{0\}.$$
\\
\textbf{Weak no-resonance case.}
In this case $\mu_1=\dots =\mu_r$, then by Lemma \ref{le1} item $(iii)$ we have 
$$b=\mu  ^{m } \mu_{r+1}^{m_{r+1}}\dots \mu_n ^{m_n}.$$
 Then from (\ref{equa6}) we get
\begin{center}
$  \mu^{\alpha_1+\dots +\alpha_r}\mu_{r+1}^{\alpha_{r+1}} 
\dots  \mu_ n ^{\alpha_n}=  \mu_{i_1}^{-1}\dots \mu_{i_k} ^{-1} 
\mu  ^{m } \mu_{r+1}^{m_{r+1}}\dots \mu_n ^{m_n}, $
\end{center}
for all  $\alpha =(\alpha_1, \alpha_2, \dots ,\alpha_n)\in \mathbb{N}^n$ and  $i_1<\dots< i_k$ such that   $ c_{\alpha}^{i_1\dots i_k}\neq 0$.  Since there is not non-trivial relation between  $\mu,\mu_{r+1},\dots,\mu_n$'s  we have:\\  for  $s\in \{0,\dots,k\}$   such that  $i_1,\dots ,i_{s} \leq r$ and $r+1\leq i_{s+1},\dots, i_k $,  then $$\alpha_1+\dots+\alpha_r= m-s$$  and $(\alpha_{r+1},  \dots,\alpha_n)=(m_{r+1},\dots, m_{i_{s+1}}-1,\dots,m_{i_k}-1,\dots, m_n)$. Therefore
\begin{equation}\label{weak} 
g_{i_1\dots i_k}=
z_{r+1} ^{m_{r+1}}\dots z_{i_{s+1}} ^{m_{i_{s+1}}-1}\dots z_{i_k} ^{m_{i_k}-1} \dots z_n ^{m_n} \cdot\sum\limits _{ \alpha_1+\dots +\alpha_r=m-s} c_{\alpha}^{i_1\dots  i_k} z_1 ^{\alpha_1}\dots 
z_r ^{\alpha_r}\end{equation}

By the same reason in the no-resonance case, we must  have  that $m_i\in\{0,1\}$, for all $i\geq r+1$. Then if the coefficient $g_{i_1\dots  i_k} \not\equiv 0$, we  have    from   (\ref{weak}) that 
\begin{equation}\label{weak2} 
g_{i_1\dots i_k}=z_{j_1} ^{m_{j_1}}\dots z_{j_{n-r-k+s}} ^{m_{j_{n-r-k+s}}} \cdot  
\sum\limits _{ \alpha_1+\dots +\alpha_r=m-s} c_{\alpha}^{i_1\dots  i_k} z_1 ^{\alpha_1}\dots z_r ^{\alpha_r}
\end{equation} 
where $\{j_1,\dots,j_{n-r-k+s}\}=\{r+1, \dots,n\}\setminus \{i_{s+1}, \dots,i_k\}$ and $m_{i_{s+1}}=\cdots=m_{i_k}=1$.

\end{proof}

\subsection{Examples}
\begin{example} Let $X= (\mathbb{C}^5- \{0\})/<\mu>$ be a classical Hopf manifold. The $2$-form  
$$\omega= z_3\,dz_1\wedge dz_2 + z_1 \,dz_1\wedge dz_3+   z_2 \,dz_1\wedge dz_4 +z_5 \,dz_1\wedge dz_5+  z_4 \,dz_4\wedge dz_5$$ induces a regular codimension two Pfaff system  on $X$.

\end{example} 

\subsection{Examples}
\begin{example} Let $X= (\mathbb{C}^5- \{0\})/<\mu_1\cdots \mu_5>$ be a no-resonance  Hopf manifold. The $3$-form  
$$\omega= z_4z_5\,dz_1\wedge dz_2\wedge dz_3 +  z_1z_3\,dz_2\wedge dz_4\wedge dz_5$$ induces a singular   codimension three  Pfaff system  on $X$.

\end{example} 

\begin{example} Let $X= (\mathbb{C}^6-  \{0\})/<\mu \mu \mu\mu_4\mu_5\mu_6>$ be a weak no-resonance  Hopf manifold. The $2$-form  
$$ \omega= z_5z_6 \,dz_1\wedge dz_4+ z_2z_4dz_5\wedge dz_6 $$ 
induces a  singular codimension two Pfaff system  on $X$.

\end{example}

\section{Proof  of  Theorem  \ref{integrab}  }

Let $X$ be a no-resonance  Hopf manifold of dimension $n$.

\hfill

\begin{definition}
Let $\mathcal{F}$ be  a  distribution on a Hopf manifold
$X=\C^n -\{0\}/\langle  f \rangle$ with tangent sheaf $T\mathcal{F}\subset TX$,
and $G\subset GL(n,\C)$ a subgroup.
We say that $\mathcal{F}$ is {\bf $G$-invariant
distribution} if the pullback of $T\mathcal{F}$ to $\C^n$
is invariant with respect to the natural
$G$-action on $\C^n$.
\end{definition}

\hfill

\begin{proposition}\label{integral}
Let $G=(\C^*)^n\subset GL(n,\C)$ be the diagonal
subgroup, and $\F\subset T\C^n$ a $G$-invariant distribution.
Then $\F$ is integrable.
\end{proposition}

 \proof
 Since the Frobenius form, representing the obstruction
 to integrability, is continuous, it would suffice to prove 
 \ref{integral} is general point of $\C^n$. Therefore,
 it would suffice to prove it on an open orbit of $G$,
 which is identified with $G$. A left-invariant distribution
 $B\subset TG$ is integrable. Indeed, $B$ is generated by
 left-invariant vector fields, which commute, because
 $G$ is commutative.
 \endproof

The integrability of the distribution  will follow from Proposition \ref{integral}  and the  
following theorem.

\hfill

\begin{teo}
Let $\F$ be a holomorphic distribution on a   no-resonance  Hopf
manifold  $X= \C^n -\{0\}/\langle  f \rangle$. Denote by $G$ the group $(\C^*)^n$ of diagonal matrices
commuting with $f$. Then $\F$ is $G$-invariant.
\end{teo}
\hfill

{\bf Proof. Step 1:}
Let $G_f$ be the Zariski closure of the group
$\langle f \rangle\cong \mathbb{Z}$. This is an algebraic group,
which is obtained as the smallest algebraic group containing
$\langle f \rangle$. This implies that $G_f\subset G$.
However, any algebraic subgroup of $G$ is given by
a set of equations of form $\prod_i z_i^{n_i}=1$,
hence none of them contains $\langle f \rangle$.
We have shown that $G_f=G$.

\hfill

{\bf Step 2:}
 Distributions on a manifold $X$ can be interpreted 
as subvarieties in  the total space $TX$ of a tangent bundle,
which are closed under the natural algebraic operations (addition
and multiplication by a number).

 Let $J\subset \mathcal{O}_{TX}$ be the
ideal corresponding to $\F$. Denote by $\tilde J\subset  \mathcal{O}_{\C^{2n}}$
the ideal $j_* \pi^*(J)$, where $\pi:\;  \C^n -\{0\}  \to X$
is the quotient map and $j:\; \ \C^n -\{0\}  \to \C^n$
the standard embedding. Then $\tilde J$ is a $f$-invariant
ideal sheaf in $ \mathcal{O}_{\C^{2n}}$. 

To finish the proof it remains to show
that any $\langle f \rangle$-invariant ideal in
$\mathcal{O}_{TM}$ is invariant with respect to the algebraic closure of
$\langle f \rangle$. 
This is implied by the following lemma,
which finishes the proof.
\endproof

\begin{lema}
Let $f: \C^n \to \C^n$ be a linear, invertible
holomorphic contraction,
and $I\subset \mathcal{O}_{\C^{2n}}$ a $f$-invariant ideal. Then
$I$ is $G$-invariant, where $G$ is Zariski closure of $f$
in $GL(n, \C)$.
\end{lema}
\proof
Let us call a holomorphic function on $\C^{n}$
{\bf $f$-finite} if it is contained in finite-dimensional,
$f$-invariant subspace in $H^0(\C^{2n}, \mathcal{O}_{\C^{2n}})$.
As follows from \cite[Theorem 4.2]{_OV:embedding_RS_},
$f$-finite functions are dense in the space of all
holomorphic functions on $\C^n$. 

The same argument also
proves that the $f$-finite sections of
$I$ are dense in the space of all sections of $I$.
However, by definition of Zariski closure, any
finite-dimensional $f$-invariant subspace is also
$G$-invariant, hence $I$ is also $G$-invariant.

\endproof

\end{document}